\documentclass[secnum,leqno,12pt]{article}
\usepackage{amssymb, amsmath}
\usepackage{amsthm}
\usepackage{geometry}
\usepackage{graphics}
\usepackage{indentfirst}
\usepackage[dvips]{graphicx}
\usepackage{eepic}
\usepackage{epic}
\newtheorem{thm}{Theorem}[section]
\newtheorem{cor}[thm]{Corollary}

\newtheorem{lmm}[thm]{Lemma}

\newtheorem{prp}[thm]{Proposition}

\theoremstyle{remark}
\newtheorem*{rmk*}{Remark}
\newtheorem*{ack}{Acknowledgement}
\newtheorem*{key}{Keywords}

\makeatletter

\@addtoreset{equation}{section}
\makeatother
\title{The Fourier coefficients of the McKay-Thompson series and the traces of CM values}
\author{\textsc{Toshiki Matsusaka}\footnote{Graduate School of Mathematics, Kyushu University, Motooka 744, Nishi-ku Fukuoka 819-0395, Japan.\newline e-mail: \texttt{toshikimatsusaka@gmail.com}}}

\begin{document}
\date{}

\maketitle

\begin{abstract}
The elliptic modular function $j(\tau)$ enjoys many beautiful properties. Its Fourier coefficients are related to the Monster group, and its CM values generate abelian extensions over imaginary quadratic fields. Kaneko gave an arithmetic formula for the Fourier coefficients expressed in terms of the traces of the CM values. In this article, we are concerned with analogues of Kaneko's result for the McKay-Thompson series of square-free level.
\end{abstract}

\begin{key}
\textit{elliptic modular function; Fourier coefficients; moonshine; complex multiplication; Jacobi forms.}
\end{key}

\section{Introduction}%??????

Let $d$ be a positive integer such that $-d$ is congruent to $0$ or $1$ modulo $4$, and $\mathcal{Q}_d$ the set of positive definite binary quadratic forms $Q(X,Y) = [a,b,c] := aX^2+bXY+cY^2\ (a,b,c \in \mathbb{Z})$ of discriminant $-d$. The group $\mathrm{SL}_2(\mathbb{Z})$ acts on $\mathcal{Q}_d$ by $Q \circ [\begin{smallmatrix}\alpha & \beta \\\gamma & \delta \end{smallmatrix}] = Q(\alpha X + \beta Y, \gamma X + \delta Y)$. For each $Q \in \mathcal{Q}_d$, we define the corresponding CM point $\alpha_Q$ as the unique root in the upper half-plane $\mathfrak{H}$ of $Q(X,1)=0$. We write $\Gamma_Q$ for the stabilizer of $Q$ in a group $\Gamma$. Let $j(\tau)$ ($\tau \in \mathfrak{H}$) be the elliptic modular function with Fourier expansion
\[ j(\tau) = \frac{1}{q} + 744 + 196884q + 21493760q^2 + 864299970q^3+ \cdots, \]
where $q := e^{2\pi i \tau}$. For a positive integer $m$, let $\varphi_m(j)$ be the unique polynomial in $j$ satisfying $\varphi_m(j(\tau)) = q^{-m}+O(q)$. For each $m$, we define the modular trace function by
\[ \textbf{t}_m(d) := \sum_{Q \in \mathcal{Q}_d/\mathrm{SL}_2(\mathbb{Z})}\frac{1}{|\mathrm{PSL}_2(\mathbb{Z})_Q|}\varphi_m(j(\alpha_Q)). \]
Zagier \cite[Theorem 1, 5]{Zagier02} showed that the generating function
\begin{equation}\label{Zagier}
g_m(\tau) := \sum_{\substack{d>0 \\ -d \equiv 0, 1 (4)}} \textbf{t}_m(d)q^d + 2\sigma_1(m) - \sum_{\kappa | m} \kappa q^{-\kappa^2}
\end{equation}
is a weakly holomorphic modular form of weight $3/2$ for the congruence subgroup $\Gamma_0(4)$, where $\sigma_1(n) := \sum_{d|n}d$. By virtue of this theorem, Kaneko \cite{Kaneko96} established an identity among modular forms of weight 2,
\begin{eqnarray*}
\frac{1}{2\pi i}\frac{d}{d\tau} j(\tau) = \frac{1}{2}((g_2 \theta_0)|U_4)(\tau),
\end{eqnarray*}
where $\theta_0(\tau) := \sum_{r \in \mathbb{Z}} q^{r^2}$, and $U_t$ is the operator $\sum a_n q^n \mapsto \sum a_{tn} q^n$ preserving modularity. In particular, the Fourier coefficients on both sides coincide. Let $c_n$ be the $n$th Fourier coefficient of $\varphi_1(j(\tau)) = j(\tau)-744$, then we have
\begin{equation}\label{Kaneko}
2n c_n = \sum_{r \in \mathbb{Z}} \textbf{t}_2(4n-r^2)
\end{equation}
for any $n \geq -1$ where $\textbf{t}_2(0)=6$, $\textbf{t}_2(-1)=-1$, $\textbf{t}_2(-4)=-2$, and $\textbf{t}_2(d)=0$ for all other negative $d$. Note that this sum is finite.\\

On the other hand, the Fourier coefficients of the $j$-function are related to the degrees of irreducible representations of the Monster group $\mathbb{M}$, the largest of the sporadic simple groups. This is known as monstrous moonshine. The first few observations are
\begin{eqnarray*}
c_1 &=& 196884 = 1 + 196883,\\
c_2 &=& 21493760 = 1+ 196883 + 21296876,\\
c_3 &=& 864299970 = 2 \times 1 + 2 \times 196883 + 21296876 + 842609326,
\end{eqnarray*}
where the sequence $\{1,\ 196883,\ 21296876,\ 842609326,\ \dots\}$ consists of degrees of irreducible representations of the Monster group. Conway and Norton \cite{CN79} formulated the monstrous moonshine conjecture as follows.
\begin{itemize}
\item There exists a graded infinite-dimensional $\mathbb{M}$-module
\[ V^{\natural} = \bigoplus_{n=-1}^{\infty} V_n^{\natural} \]
which satisfies $\mathrm{dim}V_n^{\natural} = c_n$ for $n \geq -1$. It is called the monster module.
\item For each $g \in \mathbb{M}$, we define the McKay-Thompson series
\[ T_g(\tau) := \sum_{n=-1}^{\infty} \mathrm{Tr}(g|V_n^{\natural}) q^n. \]
Then there exists a genus 0 subgroup $\Gamma_g \subset \mathrm{SL}_2(\mathbb{R})$ such that $T_g(\tau)$ is a hauptmodul on $\Gamma_g$. In other words, The fields $A_0(\Gamma_g)$ of modular functions on $\Gamma_g$ is generated by $T_g$, that is, $A_0(\Gamma_g)=\mathbb{C}(T_g)$.
\end{itemize}
In 1992, R. Borcherds \cite{Bor92} proved this conjecture.
\begin{rmk*}
(i) For the identity element $e \in \mathbb{M}$, we have $T_e(\tau) = j(\tau) - 744$.\\
(ii) For other McKay-Thompson series, similar connections are observed. (See \cite[Section 7.3: More Monstrous Moonshine]{Gannon}.) For instance, the Fourier coefficients of
\[ T_{2A}(\tau) := \frac{1}{q} + 4372q + 96256q^2 + 1240002q^3 + \cdots \]
can be expressed in terms of the degrees of irreducible representations of the Baby Monster group, that is, $4372 = 1+ 4371$, $96256 = 1 + 96255$, $1240002 = 2 \times 1 + 4371 + 96255 + 1139374,\ \dots$, where the sequence $\{1,\ 4371,\ 96255,\ 1139374,\ \dots\}$ consists of degrees of irreducible representations of the Baby Monster group.\\
\end{rmk*}

In this paper, we are concerned with the analogues of Kaneko's formula (\ref{Kaneko}) for the McKay-Thompson series of level $N$ such that $N$ is a square-free integer and the genus of the congruence subgroup $\Gamma_0(N)$ is 0, that is, $N$ = 2, 3, 5, 6, 7, 10, and 13. (Kaneko's result is the case of $N=1$.) For these $N$, let $\Gamma_0^*(N)$ be the Fricke group, which is generated by $\Gamma_0(N)$ and all Atkin-Lehner involutions $W_e$ for $e$ such that $e|N$ and $(e, N/e)=1$. Here $W_e$ is a matrix of the form $\frac{1}{\sqrt{e}} [\begin{smallmatrix}xe & y \\zN & we \end{smallmatrix}]$ with $\mathrm{det}W_e = 1$ and $x, y, z, w\in \mathbb{Z}$. Let $d$ be a positive integer such that $-d$ is congruent to a square modulo $4N$. We denote by $\mathcal{Q}_{d, N} := \{[a,b,c] \in \mathcal{Q}_d\ |\ a \equiv 0 \pmod{N} \}$ on which $\Gamma_0^*(N)$ acts. Moreover, we fix an integer $h \pmod{2N}$ with $h^2 \equiv -d \pmod{4N}$ and denote by $\mathcal{Q}_{d,N,h} := \{ [a,b,c] \in \mathcal{Q}_{d,N}\ |\ b \equiv h \pmod{2N} \}$ on which $\Gamma_0(N)$ acts. For genus zero groups $\Gamma_0(N)$ and $\Gamma_0^*(N)$, the corresponding hauptmoduln $j_N(\tau)$ and $j_N^*(\tau)$ can be described by means of the Dedekind $\eta$-function $\eta(\tau) := q^{1/24} \prod_{n=1}^{\infty} (1-q^n)$,
\begin{eqnarray*}
j_p(\tau) &=& T_{pB}(\tau) = \biggl( \frac{\eta(\tau)}{\eta(p\tau)} \biggr)^{\frac{24}{p-1}} + \frac{24}{p-1}\ \ \ \ \ \ \ \ \ \ \ \ \ \ \ \ \ \ \ \ \ \ \ \ \ \ \ \ (N = p = 2, 3, 5, 7, 13),\\
j_p^*(\tau) &=& T_{pA}(\tau) = \biggl( \frac{\eta(\tau)}{\eta(p\tau)} \biggr)^{\frac{24}{p-1}} + \frac{24}{p-1} + p^{\frac{12}{p-1}}\biggl( \frac{\eta(p\tau)}{\eta(\tau)} \biggr)^{\frac{24}{p-1}} \ \ (N = p = 2, 3, 5, 7, 13),\\
j_6(\tau) &=& T_{6E}(\tau) = \biggl( \frac{\eta(2\tau)\eta(3\tau)^3}{\eta(\tau)\eta(6\tau)^3} \biggr)^3-3,\\
j_6^*(\tau) &=& T_{6A}(\tau) = \biggl( \frac{\eta(\tau)\eta(3\tau)}{\eta(2\tau)\eta(6\tau)} \biggr)^6 + 6 +2^6\biggl( \frac{\eta(2\tau)\eta(6\tau)}{\eta(\tau)\eta(3\tau)} \biggr)^6,\\
j_{10}(\tau) &=& T_{10E}(\tau) = \biggl( \frac{\eta(2\tau)\eta(5\tau)^5}{\eta(\tau)\eta(10\tau)^5} \biggr)-1,\\
j_{10}^*(\tau) &=& T_{10A}(\tau) = \biggl( \frac{\eta(\tau)\eta(5\tau)}{\eta(2\tau)\eta(10\tau)} \biggr)^4 + 4 +2^4\biggl( \frac{\eta(2\tau)\eta(10\tau)}{\eta(\tau)\eta(5\tau)} \biggr)^4.
\end{eqnarray*}

\noindent For a weakly holomorphic modular function $f$ on $\Gamma_0(N)$, we define a modular trace function by
\[\textbf{t}_f^{(h)}(d) := \sum_{Q \in \mathcal{Q}_{d,N,h}/\Gamma_0(N)} \frac{1}{|\overline{\Gamma_0(N)}_Q|}f(\alpha_Q),\]
and for a weakly holomorphic modular function $f$ on $\Gamma_0^*(N)$, we define another trace function by
\[\textbf{t}_f^*(d) := \sum_{Q \in \mathcal{Q}_{d,N}/\Gamma_0^*(N)} \frac{1}{|\overline{\Gamma_0^*(N)}_Q|}f(\alpha_Q),\]
where $\overline{\Gamma} := \Gamma/\{\pm I\}$. Note that $\textbf{t}_f^{(h)}(d)$ is independent of the choice of $h$ for above $N$ in the particular case of $f \in A_0(\Gamma_0^*(N))$, then we can write $\textbf{t}_f(d) = \textbf{t}_f^{(h)}(d)$ simply. Moreover in the special case of $f = \varphi_m(j_N^*)$, it is the unique polynomial in $j_N^*$ satisfying $\varphi_m(j_N^*(\tau)) = q^{-m} + O(q)$, we put $\textbf{t}_m^{(N)}(d) := \textbf{t}_f(d)$ and $\textbf{t}_m^{(N*)}(d) := \textbf{t}_f^*(d)$. Ohta \cite{Ohta09} and the author and Osanai \cite{MO17} obtained the analogues of Kaneko's formula (\ref{Kaneko}) in the cases of $N = p =$ 2, 3, and 5, first found experimentally, and then showed the coincidence of $q$-series by using the Riemann-Roch theorem. In the present paper, we use the theory of Jacobi forms to generalize (\ref{Kaneko}). Let $c_n^{(N)}$ and $c_n^{(N*)}$ be the $n$th Fourier coefficients of $j_N(\tau)$ and $j_N^*(\tau)$, respectively.

\begin{thm}\label{Main}
For any $n \geq -1$, we have
\begin{eqnarray*}
2nc_n^{(p)} &=& \sum_{r \in \mathbb{Z}} \textbf{t}_2^{(p*)}(4n-r^2) + \frac{24(3-p \sigma_1(2/p))}{p-1} \sigma_1^{(p)}(n) \ \ \ \ \ \ \ (p = 2, 3, 5, 7, 13),\\
2nc_n^{(6)} &=& \sum_{r \in \mathbb{Z}} \textbf{t}_2^{(6*)}(4n-r^2) + 7\sigma_1^{(6)}(n) + 26\sigma_1^{(3)}(n/2) - 3\sigma_1^{(2)}(n/3),\\
2nc_n^{(10)} &=& \sum_{r \in \mathbb{Z}} \textbf{t}_2^{(10*)}(4n-r^2) + 4\sigma_1^{(10)}(n) + 12\sigma_1^{(5)}(n/2),\\
2nc_n^{(p*)} &=& \sum_{r \in \mathbb{Z}} \Biggl\{ \textbf{t}_2^{(p*)}(4n-r^2) - \textbf{t}_2^{(p*)}(4pn-r^2)\Biggr\}\ \ \ \ \ \ \ \ \ \ \ \ \ \ \ (p = 2, 3, 5, 7, 13),\\
2nc_n^{(p_1p_2*)} &=& \sum_{r \in \mathbb{Z}} \Biggl\{ \textbf{t}_2^{(p_1p_2*)}(4n-r^2) - \textbf{t}_2^{(p_1p_2*)}(4p_1n-r^2)\\
& &\ \ \ \ \ \ \ \ \ \ -\textbf{t}_2^{(p_1p_2*)}(4p_2n-r^2) + \textbf{t}_2^{(p_1p_2*)}(4p_1p_2n-r^2)\Biggr\}\ \ (p_1p_2 = 6, 10),
\end{eqnarray*}
where $\sigma_1^{(N)}(n) := \sum_{\substack{d|n \\ d \not \equiv 0 (N)}} d$ for a positive integer $n$. If $x \not \in \mathbb{Z}_{\geq0}$, the value of $\sigma_1^{(N)}(x)$ is 0, and we put $\sigma_1^{(N)}(0) := (N-1)/24$. Furthermore, we define additional values as follows,
\begin{eqnarray*}
\textbf{t}_2^{(N*)}(0) := \left\{\begin{array}{ll}5\ \ \ \ \ N=2, \\3\ \ \ \ \ N=3,5,7,13, \\5/2\ \ N=6,10,\end{array}\right.\ \textbf{t}_2^{(N*)}(-1) := -1,\ \textbf{t}_2^{(N*)}(-4) := -2,
\end{eqnarray*}
and $\textbf{t}_2^{(N*)}(d):=0$ for other negative $d$.
\end{thm}

\begin{rmk*}
By virtue of relations between $\textbf{t}_1^{(N*)}(d)$ and $\textbf{t}_2^{(N*)}(d)$, (see \cite{Kaneko96}, \cite{Kim04}, and \cite{Zagier02}), these formulas can be interpreted as the sum of $\textbf{t}_1^{(N*)}(d)$.
\end{rmk*}

The outline of this paper is as follows. In Section \ref{sec2} and \ref{sec3}, we give a review of the theory of Jacobi forms \cite{EZ85} and the work of Bruinier and Funke \cite{BF06}. In Section \ref{sec4} we prove Theorem \ref{Main}.\\

\begin{ack}
The author is grateful to Professor Masanobu Kaneko and Professor Ken Ono for carefully reading the manuscript and helpful comments.
\end{ack}

\section{The theory of Jacobi forms}\label{sec2}%??????
In this section, we follow the expositions in \cite{EZ85}. Let $k$ and $m$ be integers. For a function $\phi : \mathfrak{H} \times \mathbb{C} \rightarrow \mathbb{C}$, we define slash operators by
\begin{eqnarray*}
\biggl( \phi |_{k,m} \left[\begin{array}{cc}a & b \\c & d \end{array}\right] \biggr) (\tau, z) &:=& (c\tau + d)^{-k} e^{-2\pi i m \frac{cz^2}{c\tau + d}} \phi \biggl(\frac{a\tau+b}{c\tau+d}, \frac{z}{c\tau+d} \biggr),\ \ \ \left[\begin{array}{cc}a & b \\c & d \end{array}\right] \in \mathrm{SL}_2(\mathbb{Z}),\\
(\phi |_m [\lambda\ \mu])(\tau, z) &:=& e^{2\pi i m(\lambda^2\tau + 2\lambda z)} \phi(\tau, z + \lambda \tau + \mu),\ \ \ [\lambda\ \mu] \in \mathbb{Z}^2.
\end{eqnarray*}
A weakly holomorphic Jacobi form of weight $k$ and index $m$ is a holomorphic function $\phi : \mathfrak{H} \times \mathbb{C} \rightarrow \mathbb{C}$ satisfying
\begin{itemize}
\item $\phi|_{k,m} M = \phi\ \ (M \in \mathrm{SL}_2(\mathbb{Z}))$,
\item $\phi|_m X = \phi\ \ (X \in \mathbb{Z}^2)$,
\item $\phi$ has a Fourier expansion of the form
\[ \phi(\tau,z) = \sum_{\substack{n \gg -\infty \\r\in\mathbb{Z}}} c(n,r) q^n \zeta^r,\ \ (q := e^{2\pi i \tau},\ \zeta := e^{2\pi i z}), \]
\end{itemize}
where the coefficients $c(n,r)$ depend only on the value of $4mn-r^2$ and on the class of $r \pmod{2m}$, that is, we can write as $c(n,r) = c_r(4mn-r^2)$, and it holds $c_{r'}(D) = c_r(D)\ \ (r' \equiv r\pmod{2m})$. This property gives us coefficients $c_{\mu}(D)$ for all $\mu \in \mathbb{Z}/2m\mathbb{Z}$ and all integers $D$ satisfying $D \equiv -\mu^2 \pmod{4m}$, namely
\[ c_{\mu}(D) := c\biggl( \frac{D+r^2}{4m}, r\biggr),\ \ (r \in \mathbb{Z},\ r \equiv \mu\ (\mathrm{mod}\ 2m)). \]
For $D \not \equiv -\mu^2 \pmod{4m}$, we define $c_{\mu}(D) =0$, and set
\[ h_{\mu}(\tau) := \sum_{D \gg -\infty} c_{\mu}(D) q^{D/4m},\ \ \ (\mu \in \mathbb{Z}/2m\mathbb{Z}). \]
In addition, we put the theta functions
\[ \vartheta_{m, \mu}(\tau, z) := \sum_{\substack{r\in\mathbb{Z} \\ r \equiv \mu\ (\mathrm{mod}\ 2m)}} q^{r^2/4m}\zeta^r,\ \ \ (\mu \in \mathbb{Z}/2m\mathbb{Z}), \]
then $\phi$ has the following decomposition,
\begin{equation}\label{ThetaDecomp}
\phi(\tau,z) = \sum_{\mu = 0}^{2m-1} h_{\mu}(\tau) \vartheta_{m,\mu}(\tau,z).
\end{equation} According to \cite[Section 5]{EZ85}, $h_{\mu}$ and $\vartheta_{m,\mu}$ satisfy the following transformation laws,
\begin{equation}\label{modula}
\begin{split}
h_{\mu}(\tau+1) &= e^{-2\pi i \frac{\mu^2}{4m}} h_{\mu}(\tau),\\
h_{\mu}\bigl(-\frac{1}{\tau}\bigr) &= \frac{\tau^k}{\sqrt{2m\tau/i}} \sum_{\nu=0}^{2m-1} e^{2\pi i \frac{\mu \nu}{2m}} h_{\nu}(\tau),\\
\vartheta_{m,\mu}(\tau+1,z) &= e^{2\pi i \frac{\mu^2}{4m}}\vartheta_{m,\mu}(\tau,z),\\
\vartheta_{m,\mu}\bigl(-\frac{1}{\tau}, \frac{z}{\tau}\bigr) &= \sqrt{\tau/2m i}\ e^{2\pi i m z^2/\tau} \sum_{\nu=0}^{2m-1} e^{-2\pi i\frac{\mu \nu}{2m}} \vartheta_{m,\nu}(\tau,z).
\end{split}
\end{equation}
Moreover we have,
\begin{thm}\label{EZtheta}
$\cite[Theorem\ 5.1]{EZ85}$ The decomposition (\ref{ThetaDecomp}) gives an isomorphism between the space of weakly holomorphic Jacobi forms of weight $k$ and index $m$ and the space of weakly holomorphic vector valued modular forms $(h_{\mu})_{\mu \pmod{2m}}$ on $\mathrm{SL}_2(\mathbb{Z})$ satisfying the above transformation laws and some cusp conditions.
\end{thm}
Finally, we show an easy lemma for a proof of Theorem \ref{Main}.

\begin{lmm}\label{sw}
Let $\phi(\tau,z)$ be a weakly holomorphic Jacobi form of even weight $k$ and index $m$. Then the map
\[ \phi(\tau,z) \mapsto \tilde{\phi}(\tau):=\tau^{-k} \sum_{\ell = 0}^{m-1} \phi\Bigl(-\frac{1}{m\tau}, \frac{\ell}{m}\Bigr) \]
sends a weakly holomorphic Jacobi form to a weakly holomorphic modular form of weight $k$ on $\Gamma_0(m)$.
\end{lmm}

\begin{proof}
First, for any $[\begin{smallmatrix}a & b \\c & d \end{smallmatrix}] \in \Gamma_0(m)$, we can easily see that
\[\Biggl(\sum_{\ell=0}^{m-1}\phi\Bigl(\tau, \frac{\ell}{m}\Bigr)\Biggr)\Bigg|_k\left[\begin{array}{cc}a & b \\c & d \end{array}\right] = \biggl(\sum_{\ell=0}^{m-1}\phi\Big|_m \biggl[\frac{c\ell}{m}\ \ 0\biggr]\biggr)\Bigl(\tau,\frac{d\ell}{m}\Bigr) = \sum_{\ell=0}^{m-1}\phi\Bigl(\tau, \frac{\ell}{m}\Bigr).\]
Next we check for any $[\begin{smallmatrix}a & b \\c & d \end{smallmatrix}] \in \Gamma_0(m)$,
\begin{eqnarray*}
\Biggl(\tau^{-k} \sum_{\ell = 0}^{m-1} \phi\Bigl(-\frac{1}{m\tau}, \frac{\ell}{m}\Bigr)\Biggr) \Bigg|_k \left[\begin{array}{cc}a & b \\c & d \end{array}\right] &=& m^{k/2}\Biggl(\sum_{\ell = 0}^{m-1} \phi\Bigl(\tau, \frac{\ell}{m}\Bigr)\Biggr) \Bigg|_k \left[\begin{array}{cc}0 & -1/m \\1 & 0 \end{array}\right]\left[\begin{array}{cc}a & b \\c & d \end{array}\right]\\
&=&m^{k/2}\Biggl(\sum_{\ell = 0}^{m-1} \phi\Bigl(\tau, \frac{\ell}{m}\Bigr)\Biggr) \Bigg|_k\left[\begin{array}{cc}-d & c/m \\mb & -a \end{array}\right]\left[\begin{array}{cc}0 & 1/m \\-1 & 0 \end{array}\right]\\
&=& \tau^{-k} \sum_{\ell = 0}^{m-1} \phi\Bigl(-\frac{1}{m\tau}, \frac{\ell}{m}\Bigr).
\end{eqnarray*}
\end{proof}

\section{Bruinier and Funke's work}\label{sec3} %??????
In this section, we give a review of Bruinier and Funke's work \cite{BF06} and Kim's result \cite{Kim07}.
\subsection{Preliminaries}
Let $N$ be a square-free positive integer and $V$ a rational vector space of dimension 3 given by
\[ V(\mathbb{Q}) := \Biggl\{ X = \left[\begin{array}{cc}x_1 & x_2 \\x_3 & -x_1 \end{array}\right] \in M_2(\mathbb{Q}) \Biggr\} \]
with a non-degenerate symmetric bilinear form $(X, Y) := -N\cdot \mathrm{tr}(XY)$. We write $q(X) := N \cdot \mathrm{det}(X)$ for the associated quadratic form. We fix an orientation for $V$ once and for all. The action of $G(\mathbb{Q}) := Spin(V) \simeq SL_2(\mathbb{Q})$ on V is given as a conjugation, namely
\[ g.X := gXg^{-1} \]
for $X \in V$ and $g \in G(\mathbb{Q})$. Let $D$ be the orthogonal symmetric space defined by
\[ D := \{ \mathrm{span}(X) \subset V(\mathbb{R})\ |\ q(X)>0 \}.\]
For each line $z = \mathrm{span}([\begin{smallmatrix}x_1 & x_2 \\-1 & -x_1 \end{smallmatrix}]) \in D$, we can define an element in $\mathfrak{H}$ by $\tau = -x_1 + i\sqrt{x_2-x_1^2}$. In particular, this is a bijective map and preserves $G(\mathbb{Q})$-action, that is, this map sends $g.z := \mathrm{span}\bigl(g. [\begin{smallmatrix}x_1 & x_2 \\-1 & -x_1 \end{smallmatrix}]\bigr)$ to $g\tau$ for any $g \in G(\mathbb{Q})$. The image of $\tau$ under the inverse map is given by $\mathrm{span}(X(\tau))$ where
\[ X(\tau) = \left[\begin{array}{cc}-(\tau+\overline{\tau})/2 & \tau\overline{\tau} \\-1 & (\tau+\overline{\tau})/2 \end{array}\right]. \]

Let $L \subset V(\mathbb{Q})$ be an even lattice of full rank and $L^\#$ the dual lattice of $L$ defined by $L^\# := \{X \in V(\mathbb{Q})\ |\ (X, Y) \in \mathbb{Z},\ ^{\forall}Y \in L\}.$ Let $\Gamma$ be a congruence subgroup of $\mathrm{Spin}(L)$ which preserves $L$ and acts trivially on the discriminant group $L^\#/L$. The set $\mathrm{Iso}(V)$ of all isotropic lines in $V(\mathbb{Q})$ corresponds to $P^1(\mathbb{Q}) = \mathbb{Q} \cup \{\infty\}$ via the bijection
\[ \psi: P^1(\mathbb{Q}) \ni (\alpha: \beta) \mapsto \mathrm{span}\Biggl(\left[\begin{array}{cc}-\alpha \beta & \alpha^2 \\-\beta^2 & \alpha \beta \end{array}\right] \Biggr) \in \mathrm{Iso}(V). \]
In particular, we put the isotropic line $\ell_{\infty} := \psi(\infty) = \mathrm{span} ([\begin{smallmatrix}0 & 1 \\0 & 0 \end{smallmatrix}])$. We orient all lines $\ell \in \mathrm{Iso}(V)$ by requiring that $\sigma_{\ell}.[\begin{smallmatrix}0 & 1 \\0 & 0 \end{smallmatrix}]$ to be a positively oriented basis vector of $\ell$, where we pick $\sigma_{\ell} \in \mathrm{SL}_2(\mathbb{Z})$ such that $\sigma_{\ell}. \ell_{\infty} = \ell$. For each $\ell \in \mathrm{Iso}(V)$, we define three positive rational numbers $\alpha_{\ell}$ , $\beta_{\ell}$, and $\varepsilon_{\ell}$. First, we pick $\alpha_{\ell} \in \mathbb{Q}_{>0}$ as the width of the cusp $\ell$, that is,
\[ \sigma_{\ell}^{-1}\Gamma_{\ell}\sigma_{\ell} = \Biggl\{ \pm \left[\begin{array}{cc}1 & k \alpha_{\ell} \\0 & 1 \end{array}\right]\ \Bigg|\ k \in \mathbb{Z} \Biggr\}, \]
where $\Gamma_{\ell}$ is the stabilizer of the line $\ell$ in $\Gamma$. Next, we pick a positively oriented vector $Y \in V(\mathbb{Q})$ such that $\ell = \mathrm{span}(Y)$ and $Y$ is primitive in $L$. Then we define $\beta_{\ell} \in \mathbb{Q}_{>0}$ by $\sigma_{\ell}^{-1}.Y =[\begin{smallmatrix}0 & \beta_{\ell} \\0 & 0 \end{smallmatrix}].$ Finally, we put $\varepsilon_{\ell} = \alpha_{\ell}/\beta_{\ell}$. Note that the quantities $\alpha_{\ell}$ , $\beta_{\ell}$, and $\varepsilon_{\ell}$ depend only on the $\Gamma$-class of $\ell$. \\

Let $M:= \Gamma \backslash D$ be the modular curve. For $X \in V(\mathbb{Q})$ with $q(X)>0$, we define the Heegner point in $M$ by $D_X := \mathrm{span}(X) \in D$, which corresponds to an imaginary quadratic irrational in $\mathfrak{H}$. For $m \in \mathbb{Q}_{>0}$ and $h \in L^\#$, $\Gamma$ acts on $L_{h,m} := \{ X \in L+h\ |\ q(X) = m\}$ with finitely many orbits. For a weakly holomorphic modular function $f$ on $\Gamma$, we define the modular trace function by
\[ \textbf{t}_f (h,m) := \sum_{X \in \Gamma \backslash L_{h,m}} \frac{1}{|\overline{\Gamma}_X|} f(D_X). \]

Next, we consider a vector $X \in V(\mathbb{Q})$ with $q(X)<0$. For such a vector $X \in V(\mathbb{Q})$, we define a geodesic $c_X$ in $D$ by
\[ c_X := \{z \in D\ |\ z \perp X \}, \]
and we put $c(X) := \Gamma_X \backslash c_X$ in $M$. If $q(X) \in -N\cdot (\mathbb{Q}^{\times})^2$, then $X$ is orthogonal to the two isotropic lines $\ell_X = \mathrm{span}(Y)$ and $\tilde{\ell}_X = \mathrm{span}(\tilde{Y})$ such that $Y$ and $\tilde{Y}$ are positively oriented and the triple $(X, Y, \tilde{Y})$ is a positively oriented basis for $V$. We say $\ell_X$ is the line associated to X, and write $X \sim \ell_X$. We now define the modular trace function for negative index. For $X \in V(\mathbb{Q})$ of negative norm $q(X) \in -N\cdot (\mathbb{Q}^{\times})^2$, we pick $m \in \mathbb{Q}_{>0}$ and $r \in \mathbb{Q}$ such that $\sigma_{\ell_X}^{-1}. X = [\begin{smallmatrix}m & r \\0 & -m \end{smallmatrix}]$. In particular, the geodesic $c_X$ is given in $D \simeq \mathfrak{H}$ by
\[ c_X = \sigma_{\ell_X} \{\tau \in \mathfrak{H}\ |\ \mathrm{Re}(\tau) = -r/2m\},\]
and we write $\mathrm{Re}(c(X)) := -r/2m$. For $k \in \mathbb{Q}_{>0}$ and a cusp $\ell$, we put $L_{h,-Nk^2,\ell} := \{ X \in L_{h,-Nk^2}\ |\ X \sim \ell \}$ on which $\Gamma_{\ell}
$ acts. By \cite[Section 4, (4.7)]{BF06}, we have
\begin{eqnarray*}
\nu_{\ell}(h, -Nk^2) := \#\Gamma_{\ell} \backslash L_{h, -Nk^2, \ell} = \left\{\begin{array}{ll}2k \varepsilon_{\ell}\ \ $if $L_{h,-Nk^2,\ell} \neq \emptyset, \\0\ \ \ \ \ $ otherwise$.\end{array}\right.
\end{eqnarray*}
A weakly holomorphic modular function $f$ on $\Gamma$ has a Fourier expansion at the cusp $\ell$ of the form
\begin{equation}\label{FE}
f(\sigma_{\ell}\tau) = \sum_{n \in \frac{1}{\alpha_{\ell}} \mathbb{Z}} a_{\ell}(n)q^n.
\end{equation}
By \cite[Proposition 4.7]{BF06}, we can define the modular trace function for negative index by
\begin{eqnarray*}
\textbf{t}_f (h, -Nk^2) := &-&\sum_{\ell \in \Gamma \backslash \mathrm{Iso}(V)} \nu_{\ell}(h, -Nk^2) \sum_{n \in \frac{2k}{\beta_{\ell}}\mathbb{Z}_{<0}} a_{\ell}(n) e^{2\pi i r n}\\
&-&\sum_{\ell \in \Gamma \backslash \mathrm{Iso}(V)} \nu_{\ell}(-h, -Nk^2) \sum_{n \in \frac{2k}{\beta_{\ell}}\mathbb{Z}_{<0}} a_{\ell}(n) e^{2\pi i r' n},
\end{eqnarray*}
where $r = \mathrm{Re}(c(X))$ for any $X \in L_{h,-Nk^2,\ell}$ and $r' = \mathrm{Re}(c(X))$ for any $X \in L_{-h,-Nk^2,\ell}$. If $m \in \mathbb{Q}_{<0}$ is not of the form $m = -Nk^2$ with $k\in \mathbb{Q}_{>0}$, we put $\textbf{t}_f(h,m) = 0$. In particular, $\textbf{t}_f(h,m) = 0$ for $m \ll 0$.\\

Finally, the modular trace function for zero index is defined by
\[\textbf{t}_f(h,0) := -\frac{\delta_{h,0}}{2\pi} \int_{M}^{reg} f(\tau) \frac{dxdy}{y^2}\ \ (\tau = x+iy), \]
where $\delta_{h,0}$ is the Kronecker delta. By \cite[Remark 4.9]{BF06}, we have
\begin{equation}\label{trace0}
\textbf{t}_f(h,0) = 4\delta_{h,0} \sum_{\ell \in \Gamma \backslash \mathrm{Iso}(V)} \alpha_{\ell} \sum_{n \in \mathbb{Z}_{\geq 0}}a_{\ell}(-n)\sigma_1(n).
\end{equation}

\subsection{Modularity of the modular trace function}%????????
\begin{thm}$\cite[Theorem\ 4.5]{BF06}$ \label{BFmain}
Let $f$ be a weakly holomorphic modular function on $\Gamma$ with Fourier expansion as in (\ref{FE}), and assume that the constant coefficients of $f$ at all cusps of $M$ vanish. Then the generating function
\[ I_h(\tau, f) := \sum_{m \geq 0}\textbf{t}_f(h,m) q^m + \sum_{k>0}\textbf{t}_f(h,-Nk^2) q^{-Nk^2} \]
satisfies the following transformation laws,
\begin{eqnarray*}
I_h(\tau+1, f) &=& e^{2\pi i \frac{(h,h)}{2}}I_h(\tau,f),\\
I_h(-\frac{1}{\tau},f) &=& \sqrt{\tau}^3\frac{\sqrt{i}}{\sqrt{|L^\#/L|}} \sum_{h' \in L^\#/L}e^{-2\pi i (h,h')}I_{h'}(\tau,f).
\end{eqnarray*}
\end{thm}

We consider some special cases. Let $p$ be a prime number, and put
\[ L = \Biggl\{ X = \left[\begin{array}{cc}b & c/p \\a & -b \end{array}\right]\ \Bigg|\ a,b,c \in \mathbb{Z} \Biggr\} \]
with $q(X) = p\cdot \mathrm{det}(X)$. Then the action of the congruence subgroup $\Gamma_0(p)$ preserves this lattice $L$. Kim \cite{Kim07} applied Theorem \ref{BFmain} and Theorem \ref{EZtheta} to this case, and obtained the following theorem.

\begin{thm}$\cite[Theorem\ 1.1]{Kim07}$ \label{Kmain}
Let $f(\tau) = \sum_na(n)q^n$ be a weakly holomorphic modular function on $\Gamma_0^*(p)$ with $a(0) = 0$. We put
\[\textbf{t}_f(0) = 2\sum_{n=1}^{\infty} a(-n)(\sigma_1(n) + p\sigma_1(n/p)), \]
and for negative $d$,
\[\textbf{t}_f(d) = \left\{\begin{array}{ll}-2^{\mu_p(\kappa)}\kappa \sum_{\kappa | m} a(-m) \ \ $if $d = -\kappa^2$ for some positive integer $\kappa, \\0\ \ \ \ \ \ \ \ \ \ \ \ \ \ \ \ \ \ \ \ \ \ \ \ \ \ \ \ \ \ \ \ \ \ \ \ \ \ $ otherwise$,\end{array}\right.\]
where $\mu_m(n)$ is the number of prime factors of $\mathrm{gcd}(m,n)$. If $\mathrm{gcd}(m,n)=1$, we put $\mu_m(n) := 0$. Then
\[ \sum_{\substack{n, r \in \mathbb{Z}}} \textbf{t}_f(4pn-r^2)q^n\zeta^r\]
is a weakly holomorphic Jacobi form of weight $2$ and index $p$.
\end{thm}

In the same way, we consider the case of the level $p_1p_2$, where $p_1$ and $p_2$ are distinct prime numbers. We put
\[ L = \Biggl\{ X = \left[\begin{array}{cc}b & c/p_1p_2 \\a & -b \end{array}\right]\ \Bigg|\ a,b,c \in \mathbb{Z} \Biggr\} \]
with $q(X) = p_1p_2\cdot \mathrm{det}(X)$. The congruence subgroup $\Gamma_0(p_1p_2)$ preserves $L$, and acts trivially on the discriminant group $L^\#/L$, which is expressed as
\[ L^\# / L = \Biggl\{L + \left[\begin{array}{cc}h/2p_1p_2 & 0 \\0 & -h/2p_1p_2 \end{array}\right]\ \Bigg|\ h = 0, 1, 2, \dots, 2p_1p_2-1 \Biggr\} \cong \mathbb{Z}/{2p_1p_2 \mathbb{Z}}. \]
There are four $\Gamma_0(p_1p_2)$-inequivalent cusps $\infty$, $0$, $1/p_1$, and $1/p_2$. They correspond to
\[ \ell_{\infty} := \mathrm{span}\Biggl(\left[\begin{array}{cc}0 & 1 \\0& 0 \end{array}\right] \Biggr),\ \ell_0 := \mathrm{span}\Biggl(\left[\begin{array}{cc}0 & 0 \\-1& 0 \end{array}\right] \Biggr),\ \ell_{1/p_1} := \mathrm{span}\Biggl(\left[\begin{array}{cc}-p_1 & 1 \\-p_1^2& p_1 \end{array}\right] \Biggr),\]
and \[\ell_{1/p_2} := \mathrm{span}\Biggl(\left[\begin{array}{cc}-p_2 & 1 \\-p_2^2& p_2 \end{array}\right] \Biggr)\]
via the bijective map $\psi$. For these isotropic lines, we can compute the quantities $\alpha_{\ell}$, $\beta_{\ell}$, and $\varepsilon_{\ell}$ as follows.
\begin{equation}\label{abc}
\begin{split}
& \alpha_{\ell_{\infty}} = 1,\ \ \beta_{\ell_{\infty}} = \frac{1}{p_1p_2},\ \ \varepsilon_{\ell_{\infty}}=p_1p_2,\\
& \alpha_{\ell_{0}} = p_1p_2,\ \ \beta_{\ell_{0}} = 1,\ \ \varepsilon_{\ell_{0}}=p_1p_2,\\
& \alpha_{\ell_{1/p_1}} = p_2,\ \ \beta_{\ell_{1/p_1}} = \frac{1}{p_1},\ \ \varepsilon_{\ell_{1/p_1}}=p_1p_2,\\
& \alpha_{\ell_{1/p_2}} = p_1,\ \ \beta_{\ell_{1/p_2}} = \frac{1}{p_2},\ \ \varepsilon_{\ell_{1/p_2}}=p_1p_2.
\end{split}
\end{equation}

A weakly holomorphic modular function $f(\tau) = \sum_na(n)q^n$ on $\Gamma_0^*(p_1p_2)$ has a Fourier expansion of the form (\ref{FE}) at each cusp $\ell$. By direct calculation, we have
\begin{equation}\label{FC}
\begin{split}
a_{\ell_{\infty}}(n) &= a(n),\\
a_{\ell_{0}}(n/p_1p_2) &= a(n),\\
a_{\ell_{1/p_1}}(n/p_2) &= e^{-2\pi i n \frac{b}{p_2}} a(n),\ \ bp_1 \equiv -1 \pmod{p_2},\\
a_{\ell_{1/p_2}}(n/p_1) &= e^{-2\pi i n \frac{b'}{p_1}} a(n),\ \ b'p_2 \equiv -1 \pmod{p_1}.
\end{split}
\end{equation}
We assume the constant term $a(0) = 0$, then we have the constant terms at all cusps vanish by (\ref{FC}). Applying Theorem \ref{BFmain} to the above case, the function $I_h(\tau, f)$ satisfies
\begin{eqnarray*}
I_h(\tau+1,f) &=& e^{-2\pi i \frac{h^2}{4p_1p_2}}I_h(\tau, f),\\
I_h(-\frac{1}{\tau}, f) &=& \frac{\tau^2}{\sqrt{2p_1p_2\tau/i}}\sum_{h' = 0}^{2p_1p_2-1}e^{2\pi i \frac{hh'}{2p_1p_2}}I_{h'}(\tau,f).
\end{eqnarray*}
By Theorem \ref{EZtheta}, we can obtain a weakly holomorphic Jacobi form of weight 2 and index $p_1p_2$. For further details, we compute the modular trace functions.

\begin{lmm}
For a positive integer $d$, we have
\[ \textbf{t}_f(h, d/4p_1p_2) = 2\textbf{t}_f(d).\]
\end{lmm}

\begin{proof}
For each vector
\[X = \left[\begin{array}{cc}b+h/2p_1p_2 & c/p_1p_2 \\-a & -b-h/2p_1p_2 \end{array}\right] \in L + h\]
with positive norm $d/4p_1p_2$, we put
\[ Q = \left[\begin{array}{cc}ap_1p_2 & bp_1p_2+h/2 \\bp_1p_2+h/2 & c \end{array}\right] = p_1p_2 \left[\begin{array}{cc}0 & -1\\1 & 0 \end{array}\right]X.\]
Then we can see that the discriminant of the corresponding binary quadratic form is $-d = (2bp_1p_2+h)^2-4acp_1p_2 = -4p_1p_2 q(X)$. Note that if $a$ is positive (resp. negative), $Q$ is positive (resp. negative) definite. Thus we have
\begin{eqnarray*}
\textbf{t}_f(h, d/4p_1p_2) &=& \sum_{X \in \Gamma_0(p_1p_2) \backslash L_{h,d/4p_1p_2}} \frac{1}{|\overline{\Gamma_0(p_1p_2)}_X|}f(D_X)\\
&=& 2\sum_{Q \in \mathcal{Q}_{d,p_1p_2,h}/\Gamma_0(p_1p_2)} \frac{1}{|\overline{\Gamma_0(p_1p_2)}_Q|}f(\alpha_Q) = 2\textbf{t}_f(d).
\end{eqnarray*}
\end{proof}

Next we compute the modular trace functions for zero or negative index. By (\ref{trace0}), (\ref{abc}), and (\ref{FC}), we have
\[ \textbf{t}_f(h,0) = 4\delta_{h,0}\sum_{n=0}^{\infty}a(-n)\Bigl\{\sigma_1(n) + p_1p_2\sigma_1(n/p_1p_2) + p_1\sigma_1(n/p_1) + p_2\sigma_1(n/p_2)\Bigr\}. \]
Thus we define
\[\textbf{t}_f(0) := 2\sum_{n=0}^{\infty}a(-n)\Bigl\{\sigma_1(n) + p_1p_2\sigma_1(n/p_1p_2) + p_1\sigma_1(n/p_1) + p_2\sigma_1(n/p_2)\Bigr\}.\]
For a negative integer $d$, we define
\[\textbf{t}_f(d) = \left\{\begin{array}{ll}-2^{\mu_{p_1p_2}(\kappa)}\kappa \sum_{\kappa | m} a(-m) \ \ $if $d = -\kappa^2$ for some positive integer $\kappa, \\0\ \ \ \ \ \ \ \ \ \ \ \ \ \ \ \ \ \ \ \ \ \ \ \ \ \ \ \ \ \ \ \ \ \ \ \ \ \ $ otherwise$,\end{array}\right.\]
in the same way of Lemma 3.3 and Lemma 3.4 in \cite{Kim07}. Therefore, by Theorem \ref{BFmain} and Theorem \ref{EZtheta}, we obtain the following theorem.

\begin{thm}\label{levelpq}
Let $f(\tau) = \sum_na(n)q^n$ be a weakly holomorphic modular function on $\Gamma_0^*(p_1p_2)$ with $a(0) = 0$. We put
\[\textbf{t}_f(0) = 2\sum_{n=1}^{\infty}a(-n)\Bigl\{\sigma_1(n) + p_1p_2\sigma_1(n/p_1p_2) + p_1\sigma_1(n/p_1) + p_2\sigma_1(n/p_2)\Bigr\}, \]
and for negative $d$
\[\textbf{t}_f(d) = \left\{\begin{array}{ll}-2^{\mu_{p_1p_2}(\kappa)}\kappa \sum_{\kappa | m} a(-m) \ \ $if $d = -\kappa^2$ for some positive integer $\kappa, \\0\ \ \ \ \ \ \ \ \ \ \ \ \ \ \ \ \ \ \ \ \ \ \ \ \ \ \ \ \ \ \ \ \ \ \ \ \ \ $ otherwise$,\end{array}\right.\]
where $\mu_m(n)$ is the number of prime factors of $\mathrm{gcd}(m,n)$. If $\mathrm{gcd}(m,n)=1$, we put $\mu_m(n) := 0$. Then
\[ \sum_{\substack{n, r\in \mathbb{Z}}} \textbf{t}_f(4p_1p_2n-r^2)q^n\zeta^r\]is a weakly holomorphic Jacobi form of weight $2$ and index $p_1p_2$.
\end{thm}

\section{Proof of Theorem \ref{Main}}\label{sec4}%??????
Throughout this section we assume $N$ = 2, 3, 5, 6, 7, 10, or 13. We apply Theorem \ref{Kmain} and Theorem \ref{levelpq} to the special modular function $f=\varphi_2(j^*_N)$. Then we obtain
\begin{cor}\label{cor}
The generating function
\[g_2^{(N)}(\tau,z) := \sum_{\substack{n \geq 0\\r\in\mathbb{Z}}}\textbf{t}_2^{(N)}(4Nn-r^2)q^n\zeta^r\]
is a weakly holormophic Jacobi form of weight $2$ and index $N$, where
\begin{eqnarray*}
\textbf{t}_2^{(N)}(0) &=& \left\{\begin{array}{ll}6\ \ \ \ (N, 2)=1, \\10\ \ \ (N, 2)=2,\end{array}\right.\\
\textbf{t}_2^{(N)}(-1) &=& -1,\\
\textbf{t}_2^{(N)}(-4) &=& \left\{\begin{array}{ll}-2\ \ \ (N, 2)=1, \\-4\ \ \ (N, 2)=2,\end{array}\right.
\end{eqnarray*}
and $\textbf{t}_2^{(N)}(d) = 0$ for other negative $d$.
\end{cor}

Note that we can obtain recursion formulas for the modular traces by applying Choi and Kim's method \cite{CK10} to this corollary. 

\begin{lmm}$\cite{Kim06}$\label{tracerel}
For a positive integer $d$, we have
\begin{eqnarray*}
\textbf{t}_2^{(N*)}(d) = 2^{-\mu_{N}(d)}\textbf{t}_2^{(N)}(d),
\end{eqnarray*}
where $\mu_N(d)$ is the number of prime factors of $\mathrm{gcd}(N,d)$.
\end{lmm}
\begin{rmk*}
This lemma works for a general weakly holomorphic modular function $f$ on $\Gamma_0^*(N)$.
\end{rmk*}
\begin{proof}
We consider only the case of prime level $N=p$. We put the Atkin-Lehner involution $W_p = \frac{1}{\sqrt{p}}[\begin{smallmatrix}0 & -1 \\ p & 0 \end{smallmatrix}]$, and let $d$ be a positive integer. We take $h \pmod{2p}$ such that $h^2 \equiv -d \pmod{4p}$, then $h$ is divisible by $p$ if and only if $p$ divides $d$. For each $Q=[a,b,c] \in \mathcal{Q}_{d,p,h}$, the quadratic form $Q \circ W_p = [cp,-b,a/p]$ is also in $\mathcal{Q}_{d,p,h}$ if and only if $p$ divides $h$, that is, $p$ divides $d$. If $d$ is not divisible by $p$, then the map $\mathcal{Q}_{d,p,h}/\Gamma_0(p) \ni [a,b,c] \mapsto [a,b,c] \in \mathcal{Q}_{d,p}/\Gamma_0^*(p)$ is bijective, thus we have $\textbf{t}_f(d) = \textbf{t}^*_f(d)$ for a modular function $f$ on $\Gamma_0^*(p)$. If $p|d$ and $[a,b,c] \neq [cp,-b,a/p]$ in $\mathcal{Q}_{d,p,h}/\Gamma_0(p)$, then the map $\mathcal{Q}_{d,p,h}/\Gamma_0(p) \ni [a,b,c], [cp,-b,a/p] \mapsto [a,b,c] \in \mathcal{Q}_{d,p}/\Gamma_0^*(p)$ is 2-1 correspondence. If $p|d$ and $Q = [a,b,c] = [cp,-b,a/p]$ in $\mathcal{Q}_{d,p,h}/\Gamma_0(p)$, then it holds $|\overline{\Gamma_0^*(p)}_Q| = 2|\overline{\Gamma_0(p)}_Q|$. Therefore in both cases, we have $\textbf{t}_f(d) = 2\textbf{t}^*_f(d)$. In the same way, we can show the case of level $N=p_1p_2$.
\end{proof}

We define the modular trace function $\textbf{t}_2^{(N*)}(d)$ for non-positive index $d$ satisfying the relation in Lemma \ref{tracerel}. By Corollary \ref{cor} and Lemma \ref{sw}, we obtain a weakly holomorphic modular form of weight 2 on $\Gamma_0(N)$.

\begin{prp}\label{propsw}
The generating function
\[G_2^{(N*)}(\tau) := \sum_{n=-1}^{\infty} \Biggl( \sum_{r\in \mathbb{Z}}\textbf{t}_2^{(N*)}(4n-r^2) \Biggr)q^n\]
is a weakly holomorphic modular form of weight $2$ on $\Gamma_0(N)$.
\end{prp}

\begin{proof}
Applying Lemma \ref{sw} to the generating function $g_2^{(N)}(\tau,z)$, we obtain a weakly holomorphic modular form of weight 2 on $\Gamma_0(N)$
\[\tilde{g}_2^{(N)}(\tau) = \frac{1}{\tau^2}\sum_{\ell=0}^{N-1}g_2^{(N)}\Bigl(-\frac{1}{N\tau},\frac{\ell}{N}\Bigr).\]
Since the Jacobi form $g_2^{(N)}(\tau, z)$ has a theta decomposition (\ref{ThetaDecomp})
\[g_2^{(N)}(\tau,z) = \sum_{\mu=0}^{2N-1} h_{\mu}(\tau) \vartheta_{N,\mu}(\tau,z),
\]
where $h_{\mu}(\tau)$ is a partial generating function
\[h_{\mu}(\tau) := \sum_{d \equiv -\mu^2\ (\mathrm{mod}\ 4N)}\textbf{t}_2^{(N)}(d)q^{d/4N},\]
the function $\tilde{g}_2^{(N)}(\tau)$ can be expressed as follows,
\[ \tilde{g}_2^{(N)}(\tau) = \frac{1}{\tau^2}\sum_{\ell=0}^{N-1}\sum_{\mu=0}^{2N-1}h_{\mu}\Bigl(-\frac{1}{N\tau}\Bigr)\vartheta_{N,\mu}\Bigl(-\frac{1}{N\tau},\frac{\ell}{N}\Bigr).\]
Note that we can easily see that
\[
\vartheta_{N,\mu}\Bigl(-\frac{1}{N\tau},\frac{\ell}{N}\Bigr) = e^{2\pi i\frac{\ell}{N}\mu} \vartheta_{N,\mu}\Bigl(-\frac{1}{N\tau},0\Bigr).
\]
By the modularity (\ref{modula}) of the functions $h_{\mu}(\tau)$ and $\vartheta_{N,\mu}(\tau,z)$, we have directly
\[\tilde{g}_2^{(N)}(\tau) = \frac{N}{2} \sum_{\mu=0}^{2N-1} \sum_{\ell=0}^{N-1}e^{2\pi i\frac{\ell}{N}\mu} \sum_{\nu=0}^{2N-1}e^{2\pi i\frac{\mu\nu}{2N}}h_{\nu}(N\tau)\sum_{n=0}^{2N-1}e^{-2\pi i\frac{\mu n}{2N}} \vartheta_{N,n}(N\tau,0).\]
Since the sum $\sum_{\ell=0}^{N-1}e^{2\pi i\frac{\ell}{N}\mu}$ is equal to $N$ or 0 according as $N|\mu$, we have
\begin{eqnarray*}
\tilde{g}_2^{(N)}(\tau) &=& \frac{N}{2} \Biggl\{ N \sum_{\nu=0}^{2N-1}h_{\nu}(N\tau)\sum_{n=0}^{2N-1}\vartheta_{N,n}(N\tau,0)\\
& &\ \ \ \ \ \ \ \ \ \ \ \ \ \ \ \ \ \ + N \sum_{\nu=0}^{2N-1}e^{\pi i\nu}h_{\nu}(N\tau)\sum_{n=0}^{2N-1}e^{-\pi in}\vartheta_{N,n}(N\tau,0)\Biggr\}\\
&=& \frac{N^2}{2}\Biggl\{\sum_{\nu=0}^{2N-1}\sum_{n=0}^{2N-1}h_{\nu}(N\tau)\vartheta_{N,n}(N\tau,0)+ \sum_{\nu=0}^{2N-1}\sum_{n=0}^{2N-1}(-1)^{\nu+n}h_{\nu}(N\tau)\vartheta_{N,n}(N\tau,0)\Biggr\}\\
&=& N^2\Biggl\{ \sum_{\nu: even}h_{\nu}(N\tau) \cdot \sum_{n: even}\vartheta_{N,n}(N\tau,0) +\sum_{\nu: odd}h_{\nu}(N\tau) \cdot \sum_{n: odd}\vartheta_{N,n}(N\tau,0) \Biggr\}\\
&=:& N^2 \bigl\{ g_2^{(N,0)}(\tau)\theta_0^{(0)}(\tau) + g_2^{(N,3)}(\tau)\theta_0^{(1)}(\tau) \bigr\}.
\end{eqnarray*}
By Lemma \ref{tracerel}, we have
\begin{eqnarray*}
g_2^{(N,0)}(\tau) &:=& \sum_{\nu: even}h_{\nu}(N\tau) = 2^{\mu_N(N)}\sum_{d \equiv 0\ (\mathrm{mod}\ 4)}\textbf{t}_2^{(N*)}(d)q^{d/4},\\
g_2^{(N,3)}(\tau) &:=& \sum_{\nu: odd}h_{\nu}(N\tau) = 2^{\mu_N(N)}\sum_{d \equiv 3\ (\mathrm{mod}\ 4)}\textbf{t}_2^{(N*)}(d)q^{d/4},\\
\theta_0^{(0)}(\tau) &:=& \sum_{n: even}\vartheta_{N,n}(N\tau,0) = \sum_{r: even}q^{r^2/4},\\
\theta_0^{(1)}(\tau) &:=& \sum_{n: odd}\vartheta_{N,n}(N\tau,0) = \sum_{r: odd}q^{r^2/4}.
\end{eqnarray*}
Then we see that
\begin{equation}\label{decom}
g_2^{(N,0)}(\tau)\theta_0^{(0)}(\tau) + g_2^{(N,3)}(\tau)\theta_0^{(1)}(\tau) = 2^{\mu_N(N)}\sum_{n=-1}^{\infty} \Biggl( \sum_{r\in \mathbb{Z}}\textbf{t}_2^{(N*)}(4n-r^2) \Biggr)q^n.
\end{equation}
Thus we conclude that
\[G_2^{(N*)}(\tau) = N^{-2} 2^{-\mu_N(N)} \tilde{g}_2^{(N)}(\tau)\]
is a weakly holomorphic modular form of weight 2 on $\Gamma_0(N)$.
\end{proof}

\begin{prp}
The function $G_2^{(N*)}(\tau)$ has a pole only at the cusp $\tau = i \infty$.
\end{prp}

\begin{proof}
By Proposition \ref{propsw}, it is sufficient to show that $G_2^{(N*)}(\tau)$ does not have any pole at all cusps except for $i\infty$. We can show this by using (\ref{decom}) and modularity (\ref{modula}). For example, we consider the case of any $N$ and the cusp $\tau = 0$. By the definition of $g_2^{(N,0)}(\tau)$ and (\ref{modula}), we have
\begin{eqnarray*}
g_2^{(N,0)}\Bigl(-\frac{1}{\tau}\Bigr) &=& \sum_{\nu: even}h_{\nu}\Bigl(-\frac{N}{\tau}\Bigr)\ =\ \sum_{\nu: even}\frac{\tau^2/N^2}{\sqrt{2\tau/i}}\sum_{\mu=0}^{2N-1}e^{2\pi i\frac{\nu\mu}{2N}} h_{\mu}\Bigl(\frac{\tau}{N}\Bigr)\\
&\underset{(\nu=2n)}{=}& \sqrt{\frac{i}{2}}\frac{\tau^{3/2}}{N^2}\sum_{\mu=0}^{2N-1}\sum_{n=0}^{N-1}e^{2\pi i\frac{n\mu}{N}}h_{\mu}\Bigl(\frac{\tau}{N}\Bigr)\ =\ \sqrt{\frac{i}{2}}\frac{\tau^{3/2}}{N} \Biggl( h_0\Bigl(\frac{\tau}{N}\Bigr) + h_N\Bigl(\frac{\tau}{N}\Bigr) \Biggr).
\end{eqnarray*}
In a similar way, we obtain
\begin{eqnarray*}
g_2^{(N,3)}\Bigl(-\frac{1}{\tau}\Bigr) &=& \sqrt{\frac{i}{2}}\frac{\tau^{3/2}}{N} \Biggl( h_0\Bigl(\frac{\tau}{N}\Bigr) - h_N\Bigl(\frac{\tau}{N}\Bigr) \Biggr),\\
\theta_0^{(0)}\Bigl(-\frac{1}{\tau}\Bigr) &=& \sqrt{\frac{\tau}{2i}}\Biggl( \vartheta_{N,0}\Bigl(\frac{\tau}{N},0\Bigr) + \vartheta_{N,N}\Bigl(\frac{\tau}{N},0\Bigr)\Biggr),\\
\theta_0^{(1)}\Bigl(-\frac{1}{\tau}\Bigr) &=& \sqrt{\frac{\tau}{2i}}\Biggl( \vartheta_{N,0}\Bigl(\frac{\tau}{N},0\Bigr) - \vartheta_{N,N}\Bigl(\frac{\tau}{N},0\Bigr)\Biggr).
\end{eqnarray*}
Therefore we have
\begin{eqnarray*}
\tau^{-2}G_2^{(N*)}\Bigl(-\frac{1}{\tau}\Bigr) &=& 2^{-\mu_N(N)}\tau^{-2}\Biggl(g_2^{(N,0)}\Bigl(-\frac{1}{\tau}\Bigr)\theta_0^{(0)}\Bigl(-\frac{1}{\tau}\Bigr) + g_2^{(N,3)}\Bigl(-\frac{1}{\tau}\Bigr)\theta_0^{(1)}\Bigl(-\frac{1}{\tau}\Bigr)\Biggr)\\
&=& 2^{-\mu_N(N)}\frac{1}{N}\Biggl( h_0\Bigl(\frac{\tau}{N}\Bigr)\vartheta_{N,0}\Bigl(\frac{\tau}{N},0\Bigr) + h_N\Bigl(\frac{\tau}{N}\Bigr)\vartheta_{N,N}\Bigl(\frac{\tau}{N},0\Bigr) \Biggr).
\end{eqnarray*}
Note that the value of the modular trace function $\textbf{t}_2^{(N)}(d)$ for negative index is zero except for $d = -1$, $-4$, and the partial generating functions $h_0(\tau/N)$ and $h_N(\tau/N)$ are given as
\begin{eqnarray*}
h_0\Bigl(\frac{\tau}{N}\Bigr) &=& \sum_{d \equiv 0\ (\mathrm{mod}\ 4N)}\textbf{t}_2^{(N)}(d)q^{d/4N^2},\\
h_N\Bigl(\frac{\tau}{N}\Bigr) &=& \sum_{d \equiv -N^2\ (\mathrm{mod}\ 4N)}\textbf{t}_2^{(N)}(d)q^{d/4N^2}.
\end{eqnarray*}
Thus if $N \neq 2$, these functions have no pole at $q = 0$, that is, $G_2^{(N*)}(\tau)$ has no pole at $\tau = 0$. If $N=2$, the pole of $h_2(\tau/2)$ at $q=0$ is canceled out by the zero of $\vartheta_{2,2}(\tau/2,0)$. In the cases of $N=6, 10$ the cusp $\tau = 1/p$ with $p|N$ can be checked similarly by direct calculation of $(p\tau+1)^{-2}G_2^{(N*)}\bigl( \frac{\tau}{p\tau+1} \bigr)$.
\end{proof}

For our $N$, the hauptmodul $j_N(\tau)$ on $\Gamma_0(N)$ also has a pole only at the cusp $\tau = i \infty$. The differential operator $(2\pi i)^{-1}\frac{d}{d\tau}$ sends a weakly holomorphic modular function to a weakly holomorphic modular form of weight 2 on the same group, then $j'_N(\tau) := (2\pi i)^{-1}\frac{d}{d\tau}j_N(\tau)$ is a weakly holomorphic modular form of weight 2 on $\Gamma_0(N)$. Canceling the pole, we obtain a holomorphic modular form
\[ 2j'_N(\tau) - G_2^{(N*)}(\tau) \in M_2(\Gamma_0(N)),\]
where $M_2(\Gamma)$ is the space of holomorphic modular forms of weight 2 on $\Gamma$. It is known that
\begin{eqnarray*}
M_2(\Gamma_0(p)) &=& \langle E_2^{(p)}(\tau) \rangle_{\mathbb{C}}\ \ \ \ \ \ \ \ \ \ \ \ \ \ \ \ \ \ \ \ \ \ \ \ \ \ \ \ \ \ \ \ \ \ \ (p = 2, 3, 5, 7, 13),\\
M_2(\Gamma_0(p_1p_2)) &=& \langle E_2^{(p_1p_2)}(\tau),E_2^{(p_2)}(p_1\tau),E_2^{(p_1)}(p_2\tau) \rangle_{\mathbb{C}}\ \ \ (p_1p_2 = 6, 10),
\end{eqnarray*}
where
\begin{eqnarray*}
E_2^{(N)}(\tau) &:=& NE_2(N\tau)-E_2(\tau) = (N-1)+24\sum_{n=1}^{\infty}\sigma_1^{(N)}(n)q^n,\\
\sigma_1^{(N)}(n) &:=& \sum_{\substack{d|n \\ d \not \equiv 0 (N)}} d.
\end{eqnarray*}
For each level $N$, we have
\[ 2j'_N(\tau)-G_2^{(N*)}(\tau) = -\sum_{r\in\mathbb{Z}}\textbf{t}_2^{(N*)}(-r^2) + \sum_{n=1}^{\infty}\Biggl\{ 2nc_n^{(N)}-\sum_{r\in\mathbb{Z}}\textbf{t}_2^{(N*)}(4n-r^2)\Biggr\}q^n.\]
By Corollary \ref{cor} and Lemma \ref{tracerel}, the constant term is given by
\begin{eqnarray*}
-\sum_{r\in\mathbb{Z}}\textbf{t}_2^{(N*)}(-r^2) = -\textbf{t}_2^{(N*)}(0)-2\textbf{t}_2^{(N*)}(-1)-2\textbf{t}_2^{(N*)}(-4) = \left\{\begin{array}{ll}1\ \ \ \ \ N=2,\\3\ \ \ \ \ N=3,5,7,13,\\7/2\ \ N=6,10.\end{array}\right.
\end{eqnarray*}
Therefore if $N=p$, we obtain that
\[ 2j'_p(\tau)-G_2^{(p*)}(\tau) = \frac{(3-p\sigma_1(2/p))}{p-1}E_2^{(p)}(\tau). \]
In the cases of $N=6, 10$, we need more terms. The first few values of modular trace functions are given by
\begin{eqnarray*}
\textbf{t}_2^{(6*)}(8) &=& \frac{1}{|\overline{\Gamma_0^*(6)}_{[6,-4,1]}|}\varphi_2(j_6^*(\alpha_{[6,-4,1]})) = \frac{1}{2}\Biggl( j_6^*\Bigl(\frac{2 + \sqrt{-2}}{6}\Bigr)^2-158\Biggr)\\
&=&\frac{1}{2}((-10)^2-158) = -29,\\
\textbf{t}_2^{(10*)}(4) &=& \frac{1}{|\overline{\Gamma_0^*(10)}_{[10,-6,1]}|}\varphi_2(j_{10}^*(\alpha_{[10,-6,1]})) = \frac{1}{4}\Biggl( j_{10}^*\Bigl(\frac{3 + \sqrt{-1}}{10}\Bigr)^2-44\Biggr)\\
&=&\frac{1}{4}((-4)^2-44) = -7,
\end{eqnarray*}
and except for the above values $\textbf{t}_2^{(N*)}(d) = 0$ for $1 \leq d \leq 8$ (when $N=6, 10$). Then we can compute the first few coefficients of $2j'_N(\tau)-G_2^{(N*)}(\tau)$, we have
\begin{eqnarray*}
2j'_6(\tau)-G_2^{(6*)}(\tau) &=& \frac{7}{2}+7q+47q^2+O(q^3)\\
&=& \frac{7}{24}E_2^{(6)}(\tau)+\frac{13}{12}E_2^{(3)}(2\tau)-\frac{1}{8}E_2^{(2)}(3\tau),\\
2j'_{10}(\tau)-G_2^{(10*)}(\tau) &=& \frac{7}{2}+4q+24q^2+O(q^3)\\
&=& \frac{1}{6}E_2^{(10)}(\tau)+\frac{1}{2}E_2^{(5)}(2\tau).
\end{eqnarray*}
Therefore we obtain the first part of Theorem \ref{Main}. By using the following relations
\begin{eqnarray*}
j_p^* &=& j_p - pj_p|U_p\ \ \ \ p = 2,3,5,7,13,\\
j_{p_1p_2}^* &=& j_{p_1p_2}-p_1j_{p_1p_2}|U_{p_1}-p_2j_{p_1p_2}|U_{p_2}+p_1p_2j_{p_1p_2}|U_{p_1p_2}\ \ \ \ p_1p_2=6,10,
\end{eqnarray*}
we obtain the second part of Theorem \ref{Main}. This concludes the proof of Theorem \ref{Main}.\\

%%%%%%%%%%%%%%%%%%%%%%%%%%%%%%%%%%%%%%%%%%%%%%%%%%

\end{document}